\documentclass[oneside,12pt]{amsart}
\usepackage{eurosym}
\usepackage{amsfonts}
\usepackage{amsmath,amssymb,amsthm,color}
\usepackage{amsopn}
\usepackage{latexsym}
\usepackage{float}
\usepackage{bm}
\usepackage[utf8]{inputenc}

\newcommand{\mn}{\mathfrak n }

\newcommand{\mz}{\mathfrak z }

\newcommand{\mv}{\mathfrak v }
\newcommand{\mh}{\mathfrak h }

\newcommand{\so}{\mathfrak{so} }

\newcommand{\bil}{g}

\newcommand{\lra}{\longrightarrow}

\newcommand{\bs}{\backslash}

\newcommand{\R}{\mathbb R}

\newcommand{\Z}{\mathbb Z}

\DeclareMathOperator{\ad}{ad}

\DeclareMathOperator{\tr}{tr}

\numberwithin{equation}{section}

 \newtheorem{teo}{Theorem}[section]
 \newtheorem{pro}[teo]{Proposition}
 
 \newtheorem{lm}[teo]{Lemma}
 \newtheorem{defi}[teo]{Definition}

 \theoremstyle{definition}
 
 \newtheorem{remark}[teo]{Remark}

\newcommand{\nc}{\newcommand}
\nc{\Iso}{\operatorname{Iso}}
 \nc{\iso}{\mathfrak{iso}}
 \nc{\sso}{\mathfrak{so}}
\nc{\Ad}{\operatorname{Ad}} 
\nc{\Sym}{\mathrm{Sym}}
  \nc{\pr}{\operatorname{pr}} 
 \nc{\Dera}{\operatorname{Dera}} \nc{\Auto}{\operatorname{Auto}}
 \nc{\LL}{{\rm L}}
\nc{\dd}{{\rm d}}
\nc{\Id}{{\rm Id}}

\begin{document}

\title{Conformal Killing forms on $2$-step nilpotent Riemannian Lie groups}
\author{Viviana del Barco}
\address{Universidad Nacional de Rosario, CONICET, 2000, Rosario, Argentina}
\email{delbarc@fceia.unr.edu.ar}

\author{Andrei Moroianu}
\address{Université Paris-Saclay, CNRS,  Laboratoire de mathématiques d'Orsay, 91405, Orsay, France}
\email{andrei.moroianu@math.cnrs.fr}

\begin{abstract} We study  left-invariant conformal Killing $2$- or $3$-forms on simply connected $2$-step nilpotent Riemannian Lie groups. We show that if the center of the group is of dimension greater than or equal to 4, then every such form is automatically coclosed (i.e. it is a Killing form). In addition, we prove that the only Riemannian 2-step nilpotent Lie groups with center of dimension at most 3 and admitting left-invariant non-coclosed conformal Killing $2$- and $3$-forms are: the Heisenberg Lie groups and their trivial 1-dimensional extensions, endowed with any left-invariant metric, and the simply connected Lie group corresponding to the free 2-step nilpotent Lie algebra on 3 generators, with a particular 1-parameter family of metrics. The explicit description of the space of conformal Killing $2$- and $3$-forms is provided in each case.
\end{abstract}

\maketitle

\section{Introduction}
Killing and conformal Killing forms on Riemannian manifolds are natural extensions to higher degrees of the notions of Killing and conformal vector fields. Killing forms were introduced in the 50's by Yano \cite{Ya} and conformal Killing forms were introduced a few years later by Kashiwada \cite{Ka} and Tachibana \cite{Ta},  but their modern systematic study only began in 2003 with the work of Semmelmann \cite{uwe}.

By definition, a conformal Killing form on a Riemannian manifold (also called twistor form or conformal Killing-Yano form by some authors) is an exterior $p$-form whose covariant derivative with respect to the Levi-Civita connection is determined by its exterior derivative and its co-differential (see Definition~\ref{d31} below). A conformal Killing form whose co-differential vanishes is called a Killing form, and in order to distinguish this case, a conformal Killing form is called strict if its co-differential is non-zero.

Most classification results for Killing and conformal Killing forms were obtained on compact Riemannian manifolds with special holonomy. More precisely, Killing forms were classified on compact Kähler manifolds \cite{Yama}, on compact symmetric spaces \cite{bms}, on compact quaternion-Kähler manifolds \cite{ms2}, and compact manifolds with holonomy $\mathrm{G}_2$ or Spin$(7)$ \cite{uwe1}. Moreover, it was shown in \cite{ms} that conformal Killing forms on Riemannian products of compact manifolds are determined by Killing forms on the factors, and in \cite{ms1}, \cite{uwe4} that conformal Killing forms on compact Kähler manifolds are necessarily of type $(p,p)$ and they are in 1-1 correspondence with the Hamiltonian 2-forms introduced by Apostolov, Calderbank, Gauduchon and Tonnesen-Friedman \cite{ACG1},  \cite{ACG2}, \cite{ACG3}.

If the manifold has generic holonomy, further geometric restrictions are required in order to obtain classification results. We mention \cite{uwe3}, where Killing forms are studied on compact nearly Kähler 6-manifolds, and \cite{a}, where compact manifolds carrying Killing 1-forms whose  exterior derivative is conformal Killing are classified. 

More recently, (conformal) Killing forms have been studied on Riemannian Lie groups, under the left-invariance assumption. In this case, the problem reduces to an algebraic system on the Lie algebra, which however is rather complicated, and classification results can be obtained only under further assumptions. 

In \cite{HO} Herrera and Origlia study 5-dimensional metric Lie algebras carrying strict conformal Killing 2-forms and obtain their classification when the center is of dimension at least $2$. In
\cite{AD} Andrada and Dotti show that a metric Lie algebra carries a strict conformal Killing 2-form whose co-differential is the metric dual of a central element if and only if  the Lie algebra is a central extension of an even-dimensional Lie algebra carrying a non-degenerate Killing 2-form, whose inverse (as endomorphism) corresponds to a closed 2-form. 

More general results can be obtained on 2-step nilpotent metric Lie algebras. Killing 2-forms on such Lie algebras were first studied in \cite{BDS}, and the classification of 2-step nilpotent metric Lie algebras admitting Killing 2-forms and 3-forms was given in \cite{va} (see also \cite{AD2} for the case of 2-forms). Killing $p$-forms for $p\ge4$ were later considered in \cite{va2}, where the classification of 2-step nilpotent metric Lie algebras carrying such objects is obtained under the additional assumption that the center of the Lie algebra has dimension at most 2. 

In the present paper we pursue our quest by studying conformal Killing forms of degree at most 3 on 2-step nilpotent metric Lie algebras. Since every conformal Killing 1-form is the metric dual of a conformal vector field, and is automatically Killing on any Lie algebra, we only study the degrees 2 and 3. 

In Theorem~\ref{teo:2forms} we show that every conformal Killing 2-form on a 2-step nilpotent metric Lie algebra is automatically Killing except on the Heisenberg Lie algebras $\mh_{2q+1}$ where, for every metric, the space of conformal Killing 2-forms is 1-dimensional (and every Killing $2$-form vanishes). Then, in Theorem~\ref{teo:3forms} we prove that the space of conformal Killing 3-forms coincides with the space of Killing $3$-forms, except for the trivial extensions $\R\oplus\mh_{2q+1}$ of  Heisenberg Lie algebras,  endowed with  arbitrary metrics, and for a 1-parameter family of metrics on the free 2-step nilpotent Lie algebra on $3$-generators $\mn_{3,2}$. In these two exceptional cases we show that the space of conformal Killing $3$-forms has dimension 2  and the space of Killing 3-forms is 1-dimensional.

\section{Preliminaries: Riemannian geometry of $2$-step nilpotent Lie groups}\label{sec:prel}
This section includes basic material on the geometry of $2$-step nilpotent Lie groups endowed with a Riemannian metric which is invariant under left-translations. 

Let $N$ be a connected Lie group and let $g$ be a left-invariant Riemannian metric on $N$. Each element $a$ of the Lie group defines a left-translation $L_a:N\lra N$ with $h\mapsto L_a(h)=ah$, which is an isometry of the metric. Therefore, $g$ is determined by its value on the tangent space at the identity, which is identified with the Lie algebra $\mn$ of $N$. Denote by $\nabla$ the Levi-Civita connection of $(N,g)$. 

For every left-invariant vector fields $X,Y,Z$ on $N$, Koszul's formula  reads
\begin{equation}\label{eq.Koszul}
g( \nabla_X Y,Z)=\frac12\{ g( [X,Y],Z)+g( [Z,X],Y)+g( [Z,Y],X)\}.
\end{equation}
This shows, in particular, that that the Levi-Civita connection preserves left-invariance, that is, the covariant derivative of a left-invariant vector field with respect to another left-invariant vector field is again left-invariant. 

From now on we will identify left-invariant vector fields  $X$ on $N$ with their values $x\in\mn$ at the identity. The Levi-Civita covariant derivative defines a linear map $\nabla:\mn\lra\mathfrak{so}(\mn)$, which by \eqref{eq.Koszul} satisfies
\begin{equation}\label{eq.Koszul1}
\nabla_xy=\frac12\left( [x,y]-\ad_x^*y-\ad_y^*x\right), \quad  x,y\in \mn,
\end{equation}
where $\ad_x^*$ denotes the adjoint of $\ad_x$ with respect to $g$.

Recall that a Lie algebra $\mn$ is said to be $2$-step nilpotent if it is not abelian and $\ad_x^2=0$ for all $x\in\mn$. Equivalently, $\mn$ is $2$-step nilpotent if its commutator $\mn':=[\mn,\mn]$ is non-trivial and is contained in the center $\mz$ of $\mn$.
 
For the rest of the section we assume that $N$ is a simply connected Lie group correspon\-ding to a $2$-step nilpotent Lie algebra $\mn$. The main geometric properties of $(N,g)$ will be described through linear objects in the metric Lie algebra $(\mn,g)$, following \cite{EB}.

Let $\mv$ be the orthogonal complement of $\mz$ in $\mn$. Since $\mn$ is 2-step nilpotent, the or\-thogonal decomposition as a sum of vector spaces $\mn=\mv\oplus\mz$ is non-trivial. Each $z\in \mz$ defines a skew-symmetric endomorphism $j(z):\mv\lra\mv$ by the formula
\begin{equation}\label{eq:jota}
g(j(z)x,y) :=g(z,[x,y]) \quad \mbox{ for all } x,y\in\mv.
\end{equation}
These endomorphisms verify
\begin{equation}
\label{int}\bigcap_{z\in\mz}\ker j(z)=0.
\end{equation}
Indeed, if $x\in \mv$ is in $\ker j(z)$ for all $z\in \mz$, then \eqref{eq:jota} implies that $x$ lies also in $\mz$, so $x=0$. Moreover, since $\mn'\subseteq\mz$, the Lie algebra structure of $\mn$ is completely determined by the map $j: \mz \lra \sso(\mv)$.

Using this linear map, we can also describe important geometric data of the Riemannian manifold $(N,\bil)$. For instance, after \eqref{eq.Koszul1}, the covariant derivative verifies
\begin{equation}\label{eq:nabla}\left\{
\begin{array}{ll}
\nabla_x y=\frac12 \,[x,y] & \mbox{ if } x,y\in\mv,\\
\nabla_x z=\nabla_zx=-\frac12 j(z)x & \mbox{ if } x\in\mv,\,z\in\mz,\\
\nabla_z z'=0& \mbox{ if } z, z'\in\mz.
\end{array}\right.
\end{equation}

We can thus observe the following behaviour of the  covariant derivative
\begin{equation}\label{nablax} \nabla_x\mv\subset\mz,\qquad\nabla_x\mz\subset\mv,\qquad \mbox{ for all } x\in\mv,
\end{equation}and 
\begin{equation}\label{nablaz} \nabla_z\mv\subset\mv,\qquad\nabla_z\mz=0,\qquad \mbox{ for all } z\in\mz.
\end{equation}

\section{Conformal Killing forms on Lie groups}
In this section we introduce our object of study, namely, conformal Killing forms on Riemannian manifolds. Afterwards, we describe basic properties of left-invariant conformal Killing forms on Riemannian Lie groups.

Let $(M,g)$ be an oriented $n$-dimensional Riemannian manifold and denote by $\nabla$ the Levi-Civita connection, by $\dd$ the exterior derivative and by $\dd^*$ its formal adjoint. Recall that, with respect to a local orthonormal frame $e_1, \ldots, e_n$ of $({\rm T}M,g)$, the operator $\dd^*$ reads
\[
\dd^*=-\sum_{i=1}^n e_i\lrcorner\ \nabla_{e_i}.
\]

\begin{defi}[\cite{uwe}]\label{d31}
A {\em conformal Killing} $k$-form on $(M,g)$ is an exterior $k$-form $\alpha$ that satisfies 
\begin{equation}\label{ck}
\nabla_Y \alpha=\frac1{k+1}Y\lrcorner\ \dd\alpha-\frac{1}{n-k+1} Y^\flat\wedge \dd^*\alpha
\end{equation}
for every vector field $Y$ in $M$. Here $Y^\flat$ denotes the metric dual $1$-form of $Y$, i.e. \linebreak $Y^\flat=g( Y, \cdot)$.
\end{defi}

By applying the Hodge duality to formula \eqref{ck} one can readily check that $\alpha$ is conformal Killing form if and only its Hodge dual $*\alpha$ is conformal Killing.

Conformal Killing forms on $(M,g)$ which are coclosed are called {\em Killing forms}. Notice that, given a vector field $X$ on $M$, $X^\flat$ is a conformal Killing (respectively Killing) $1$-form if and only if $X$ is a conformal (respectively Killing) vector field. Most of the time, we will identify vector fields with their metric dual 1-forms.

Let now $N$ be a Lie group with Lie algebra $\mn$ and let $g$ be a left-invariant metric on $N$. Left-invariant differential $k$-forms on $N$, that is, invariant under $L_a$ for all $a\in N$, can be identified to elements in $\Lambda^k\mn^*$. Since $\dd$ and $\dd^*$  preserve left-invariance, they define linear  operators on $\Lambda^*\mn^*$, which we denote with the same symbols for simplicity. In particular, the linear operator $\dd:\Lambda^k\mn^*\lra \Lambda^{k+1}\mn^*$ is the Lie algebra differential, and  $\dd^*$  is  the metric adjoint of $\dd$ as soon as $\mn$ is unimodular. An element $\alpha\in \Lambda^k\mn^*$ corresponds to a left-invariant conformal Killing form on $(N,g)$ if and only if 
\begin{equation}\label{ckinv}
\nabla_y \alpha=\frac1{k+1}\ y\lrcorner\ \dd\alpha-\frac{1}{n-k+1} y\wedge \dd^*\alpha, \quad \mbox{ for all } y\in \mn.
\end{equation}
Along the text, we will say that $\alpha\in\Lambda^k\mn^*$ is a conformal Killing form if it satisfies \eqref{ckinv}, and that it is a Killing form if it satisfies the additional condition $\dd^*\alpha=0$. We will denote by $\mathcal{CK}^k(\mn,g)\subseteq \Lambda^k\mn^*$ the space of conformal Killing $k$-forms on $(\mn,g)$ and by $\mathcal{K}^k(\mn,g)\subseteq \mathcal{CK}^k(\mn,g)$ the space of Killing $k$-forms.

Every skew-symmetric endomorphism $\mn\simeq\mn^*$ extends as a derivation of the exterior algebra. Namely, an endomorphism $f:\mn\lra\mn$ defines the derivation $f_*$ of $\Lambda^*\mn^*$ by the formula 
\begin{equation}\label{der} f_*\alpha:=\sum_{i=1}^nf(e_i)\wedge e_i\lrcorner\ \alpha,\qquad \mbox{ for all }\alpha\in \Lambda^*\mn^*,
\end{equation}
where $e_1, \ldots, e_n$ is any orthonormal basis of $(\mn,\bil)$. 

Assume that $\mn$ is 2-step nilpotent (and thus unimodular) and consider the orthogonal decomposition $\mn=\mv\oplus\mz$ as in Section~\ref{sec:prel}. One has an induced decomposition of the space of $k$-forms on $\mn$ as follows
\begin{equation}\label{dec}
\Lambda^k\mn^*=\bigoplus_{d\in\Z} \Lambda^d\mv^*\otimes \Lambda^{k-d}\mz^*,
\end{equation}
where, by convention, for every vector space $E$ and $d<0$ we set $\Lambda^dE^*:=0$. Of course the sum in \eqref{dec} is finite, having at most $k+1$ non-zero summands.
We denote by $\pi_{k,d}:\Lambda^k\mn^*\lra \Lambda^d\mv^*\otimes \Lambda^{k-d}\mz^*$ the corresponding projections. Given $\alpha\in \Lambda^k\mn^*$, we write accordingly $\alpha_d:=\pi_{k,d}(\alpha)$, so $\alpha=\sum_{d\in \Z}\alpha_d=\alpha_0+\ldots+\alpha_k$.

Let us introduce the further direct sum decomposition $\Lambda^k\mn^*=\Lambda^k_{\rm ev}\mn^*\oplus\Lambda^k\mn_{\rm odd}^*$, where
\begin{equation}\label{deceo}
\Lambda^k_{\rm ev}\mn^*:=\bigoplus_{d\ {\rm even}} \Lambda^d\mv^*\otimes \Lambda^{k-d}\mz^*,\quad\mbox{ and }\quad
\Lambda^k_{\rm odd}\mn^*:=\bigoplus_{d\ {\rm odd}} \Lambda^d\mv^*\otimes \Lambda^{k-d}\mz^*.
\end{equation}
We say that a $k$-form $\alpha$ is of {\em even $\mv$-degree} if $\alpha\in \Lambda^k_{\rm ev}\mn^*$ and of {\em odd $\mv$-degree} if $\alpha\in \Lambda^k_{\rm odd}\mn^*$. Correspondingly, we will denote by 
$$\mathcal{K}^k_{\rm ev}(\mn,g)\subseteq\mathcal{CK}^k_{\rm ev}(\mn,g)\subseteq \Lambda^k_{\rm ev}\mn^*\qquad \mbox{ and }\qquad\mathcal{K}^k_{\rm odd}(\mn,g)\subseteq\mathcal{CK}^k_{\rm odd}(\mn,g)\subseteq \Lambda^k_{\rm odd}\mn^*$$
 the spaces  Killing and conformal Killing $k$-forms on $(\mn,g)$ of even and odd $\mv$-degree. We will see later on (in Remark ~\ref{rem:uncp}) that the even and odd components of every conformal Killing form are again conformal Killing.

Since every linear form in $\mv^*$ annihilates $\mn'$,  the exterior differential $\dd$ vanishes on $\mv^*$ and maps $\mz^*$ to $\Lambda^2\mv^*$. We thus obtain that for $k\geq 0$,
\begin{equation}\label{d}\dd( \Lambda^d\mv^*\otimes \Lambda^{k-d}\mz^*)\subseteq  \Lambda^{d+2}\mv^*\otimes \Lambda^{k-d-1}\mz^*,\quad \mbox{ for all }d\in\Z,\end{equation}
and correspondingly
\begin{equation}\label{delta} \dd^*( \Lambda^d\mv^*\otimes \Lambda^{k-d}\mz^*)\subseteq  \Lambda^{d-2}\mv^*\otimes \Lambda^{k-d+1}\mz^*.\quad \mbox{ for all }d\in\Z,\end{equation}

Moreover, since $\nabla_y$ is a derivation of the exterior algebra of $\mn$ for every $y\in\mn$, using \eqref{nablax}--\eqref{nablaz} we obtain for every $d\in\Z$,
\begin{equation}\label{nablax1}\nabla_x(\Lambda^d\mv^*\otimes \Lambda^{k-d}\mz^*)\subseteq  (\Lambda^{d-1}\mv^*\otimes \Lambda^{k-d+1}\mz^*)\oplus  (\Lambda^{d+1}\mv^*\otimes \Lambda^{k-d-1}\mz^*),\quad \mbox{ for all } x\in\mv,
\end{equation}
and
\begin{equation}\label{nablaz1}\nabla_z(\Lambda^d\mv^*\otimes \Lambda^{k-d}\mz^*)\subseteq  \Lambda^{d}\mv^*\otimes \Lambda^{k-d}\mz^*,\quad\hbox{and}\quad \nabla_z(\Lambda^k\mz^*)=0,\quad \mbox{ for all } z\in\mz.
\end{equation}

\begin{pro}\label{pro:ck} Let $\mn$ be a $2$-step nilpotent Lie algebra with orthogonal decomposition $\mn=\mv\oplus\mz$. Then a $k$-form $\alpha$ on $\mn$ is a conformal Killing form  if and only if for every $x\in \mv$ and $z\in \mz$, its components with respect to the decomposition \eqref{dec} verify
\begin{eqnarray}
\pi_{k,d+1}(\nabla_x(\alpha_d+\alpha_{d+2}))&=&\frac{1}{k+1}x\lrcorner \dd \alpha_d-\frac1{n-k+1} x\wedge \dd^*\alpha_{d+2},\label{ckv}\\
\nabla_z\alpha_d&=&\frac{1}{k+1}z\lrcorner \dd\alpha_{d-2}-\frac1{n-k+1} z\wedge \dd^*\alpha_{d+2},\label{ckz}
\end{eqnarray}
for each $d\in\Z$. In particular,  if $\alpha\in \Lambda^k\mn^*$ is conformal Killing, then either $\dd^*\alpha_2=0$ or $\dim\mz =k-1$.
\end{pro}
\begin{proof}
Let $\alpha$ be a $k$-form on $\mn$ and consider its components $\alpha_d$ with respect to \eqref{dec}, for $d\in\Z$.  Suppose that $\alpha$ is a conformal Killing form. By \eqref{ckinv} we obtain
\begin{equation}\label{ckinvd}
\sum_{d\in\Z}\nabla_y\alpha_d=\frac{1}{k+1}\sum_{d\in\Z} y\lrcorner \dd\alpha_d-\frac1{n-k+1} \sum_{d\in\Z} y\wedge
\dd^*\alpha_d,\quad \mbox{ for all }y\in \mn.
\end{equation}

For every $x\in\mv$, take $y=x$ in \eqref{ckinvd} and, for each ${d\in\Z}$,  project this equation onto $\Lambda^{d+1}\mv^*\otimes \Lambda^{k-d-1}\mz^*$. By using \eqref{d}--\eqref{nablaz1}, this procedure gives 
\[
\pi_{k,d+1}(\nabla_x(\alpha_d+\alpha_{d+2}))=\frac{1}{k+1}x\lrcorner \dd \alpha_d-\frac1{n-k+1} x\wedge \dd^*\alpha_{d+2},
\] yielding to \eqref{ckv}. Similarly, for every $z\in\mz$, taking $y=z$ in \eqref{ckinvd} and projecting onto 
$\Lambda^{d}\mv^*\otimes \Lambda^{k-d}\mz^*$ we get \eqref{ckz}.

Conversely, by \eqref{d}--\eqref{nablaz1}, one can show that if \eqref{ckv} and \eqref{ckz} hold, then $\alpha$ satisfies \eqref{ckinvd} and thus it is a conformal Killing form.

For the last part of the proof, assume that $\alpha\in \Lambda^k\mn^*$ is conformal Killing. Since $\alpha_0\in \Lambda^k\mz^*$, we have $\nabla_z\alpha_0=0$ by \eqref{nablaz}. This fact together with \eqref{ckz} implies $z\wedge \dd^*\alpha_{2}=0$ for all $z\in \mz$; the latter is satisfied only when $\dd^*\alpha_2=0$ or $\dim\mz= k-1$. 
\end{proof}
\begin{remark}\label{rem:uncp}
From Proposition~\ref{pro:ck}, we can see that the system of equations that an exterior form has to satisfy in order to be a conformal Killing form is uncoupled between even $\mv$-degree and odd $\mv$-degree. This implies
\begin{equation}\label{ckdec}\mathcal{CK}^k(\mn,g)=\mathcal{CK}^k_{\rm ev}(\mn,g)\oplus \mathcal{CK}^k_{\rm odd}(\mn,g), \quad \mbox{ for all }k=1, \ldots, n.
\end{equation}
The similar decomposition for the space of Killing forms,
\begin{equation}\label{kdec}\mathcal{K}^k(\mn,g)=\mathcal{K}^k_{\rm ev}(\mn,g)\oplus \mathcal{K}^k_{\rm odd}(\mn,g), \quad \mbox{ for all }k=1, \ldots, n.
\end{equation}
was obtained in \cite[Remark 4.3]{va2} where, however, the odd and even degree components of a form are defined with respect to the degree in $\mz$ instead of $\mv$. This alternative definition interchanges the two summands in \eqref{kdec} when the total degree $k$ is odd, but does not affect the decomposition itself.
\end{remark}

\section{Conformal Killing $2$-forms on $2$-step nilpotent Riemannian Lie groups}
The main result of this section is to describe explicitly those $2$-step nilpotent metric Lie algebras admitting strict conformal Killing $2$-forms. We will see that these examples will arise within the family of Heisenberg Lie algebras.

Recall that the Heisenberg Lie algebra $\mh_{2q+1}$ of dimension $2q+1$ admits a basis\linebreak $e_1, \ldots, e_{2q},z$ such that the only non-trivial Lie brackets are
\begin{equation}\label{eq:heisb}
[e_{2i-1},e_{2i}]=z, \qquad i=1, \ldots, q.
\end{equation}
In particular, the center $\mz$ of $\mh_{2q+1}$ is 1-dimensional and spanned by $z$. Conversely, it is easy to prove that every $2$-step nilpotent Lie algebra with 1-dimensional center has odd dimension $2q+1$ and is isomorphic to the Heisenberg Lie algebra $\mh_{2q+1}$.

Given an inner product on $\mh_{2q+1}$, the orthogonal decomposition $\mn=\mv\oplus \mz$ verifies $\dim\mv=2q$ and, for every $z\in \mz$, $j(z)\in \so(\mv)$ is non-singular by \eqref{int}. 

\begin{teo}\label{teo:2forms} On any $2$-step nilpotent metric Lie algebra $(\mn,g)$, $\mathcal{CK}^2(\mn,g)=\mathcal{K}^2(\mn,g)$, except when $\mn$ is a Heisenberg Lie algebra where, for every metric $g$, $\mathcal{CK}^2(\mn,g)$ is $1$-dimensional and $\mathcal{K}^2(\mn,g)=0$.
\end{teo}
\begin{proof}
Let $(\mn,g)$ be a metric 2-step nilpotent Lie algebra and let $\alpha\in\mathcal{CK}^2(\mn,g)$. Write $\alpha=\alpha_0+\alpha_1+\alpha_2$ with respect to the decomposition \eqref{dec}. From \eqref{d} and \eqref{delta} we obtain 
\begin{equation}\label{eq:d10}
\dd\alpha=\dd\alpha_1+\dd\alpha_0\in \Lambda^3\mv^*\oplus (\Lambda^2\mv^*\otimes\mz^*)\;\mbox{ and }\;\dd^*\alpha=\dd^*\alpha_2\in \mz^*.
\end{equation}
Moreover, by the second part of Proposition~\ref{pro:ck},  either $ \dd^*\alpha_2=0$ or $\mz$ is $1$-dimensional; the former case implies that $\alpha$ is Killing by \eqref{eq:d10}. Therefore, if $\alpha$ is strict, then $\dim \mz=1$ and thus $\alpha_0=0$.

Assume that $\dim(\mz)=1$, so that $\mn$ is isomorphic to a Heisenberg Lie algebra $\mh_{2q+1}$. We shall prove that $\mathcal{CK}^2(\mh_{2q+1},g)$ is 1-dimensional, for every metric $g$.

Fix a unit vector $z\in \mz$ and write the component $\alpha_1$ of $\alpha$ as $\alpha_1=\beta\wedge z$ for some $\beta\in \mv^*$. Then \eqref{ckz} for $d=1$ yields 
$$\nabla_{z}\beta\wedge z=0,$$
which implies $\nabla_{z}\beta=0$. However, $\nabla_{z}\beta=-\frac12j({z})\beta$, and $j({z})$ is non-singular by \eqref{int}, so finally $\beta=0$.

Consequently $\alpha=\alpha_2\in \Lambda^2\mv^*$ and thus $\dd\alpha=0$. Denoting $\dd^*\alpha=c{z}$ for some $c\in\mathbb{R}$, \eqref{ckz} for $d=2$ reduces to $\nabla_{z} \alpha=0$, which by \eqref{eq:nabla} is equivalent to $j({z})_*\alpha=0$. In other words, $j({z})$ commutes with the skew-symmetric endomorphism of $\mv$ associated to $\alpha$. Moreover, notice that for every $x\in \mv$, $\nabla_x\alpha\in \mv\otimes \mz$ and
\eqref{ckv} for $d=0$ is equivalent to 
$$z\lrcorner\ \nabla_x \alpha=\frac{c}{n-1} x, \quad \mbox{ for all }x\in\mv.$$
Since ${z}\lrcorner\ \alpha=0$, this last equation 
together with \eqref{eq:nabla} imply
\[0=\nabla_x(z\lrcorner\ \alpha)=\nabla_xz\lrcorner\ \alpha +z\lrcorner\ \nabla_x\alpha=-\frac12 (j(z)x)\lrcorner\ \alpha+\frac{c}{n-1} x, \quad \mbox{ for all }x\in \mv.
\]
This implies that the endomorphism of $\mv$ associated to $\alpha$ has to be proportional to $j({z})^{-1}$; equivalently, $\alpha$ is a multiple of the 2-form $g(j({z})^{-1}\cdot,\cdot)$. 

Conversely, one can show that $\alpha:=g(j({z})^{-1}\cdot,\cdot)$ is a conformal Killing form by  reversing the above computations and using Proposition~\ref{pro:ck}.  Moreover, $\alpha$ is a strict conformal Killing form, since $\dd^*\alpha=\frac{n-1}2z$. 
Due to \eqref{d} and the fact that $\alpha\in \Lambda^2\mv^*$, we get $\dd\alpha=0$  and thus its Hodge dual $\ast \alpha$ is a Killing $2q-1$-form, for any chosen orientation of $(\mn,g)$. This concludes the proof.
\end{proof}

\begin{remark} From the last part of the proof of Theorem~\ref{teo:2forms}, one can deduce that when 
the space of conformal Killing 2-forms does not coincide with the space of Killing 2-forms, then it has the precise description: $\mathcal {CK}^2(\mh_{2q+1},g)=\ast  \mathcal{K}^{2q-1}(\mh_{2q+1},g)$.

In general, such a description of conformal Killing 2-forms does not hold on arbitrary metric Lie algebras. Indeed, in \cite[Section 6.1]{HO} the authors exhibit metric Lie algebras carrying conformal Killing 2-forms which cannot be written as linear combinations of Killing forms and Hodge duals of Killing forms.
\end{remark}

\section{Conformal Killing $3$-forms on $2$-step nilpotent Riemannian Lie groups} 
This last section aims to show that strict conformal Killing $3$-forms only appear on 1-dimensional trivial extensions of  Heisenberg Lie algebras $\R\oplus\mh_{2q+1}$ (for any metric), and on the free $2$-step nilpotent Lie algebra $\mn_{3,2}$ of dimension 6, endowed with a particular 1-parameter family of metrics.

Let $(\mn,g)$ be a 2-step nilpotent metric Lie algebra. As a consequence of the decomposition in \eqref{ckdec}, in order to describe the space $\mathcal{CK}^3(\mn,g)$ of conformal Killing 3-forms, it is enough to study the spaces of conformal Killing $3$-forms of even or odd $\mv$-degree.

\subsection{The case of even $\mv$-degree} We start by focusing our attention on $\mathcal{CK}_{\rm ev}^3(\mn,g)$. 

\begin{pro} \label{pro:3formsevenv}
For any $2$-step nilpotent metric Lie algebra $(\mn,g)$, $\mathcal{CK}_{\rm ev}^3(\mn,g)=\mathcal{K}_{\rm ev}^3(\mn,g)$, except when $\mn$ is a $1$-dimensional trivial extension of a Heisenberg Lie algebra where, for every metric $g$, $\mathcal{CK}_{\rm ev}^3(\R\oplus\mh_{2q+1},g)$ is $2$-dimensional and $\mathcal{K}_{\rm ev}^3(\R\oplus\mh_{2q+1},g)$ is $1$-dimensional.
\end{pro}

\begin{proof} Let $\alpha$ be a $3$-form on $\mn$ of even $\mv$-degree, and consider its decomposition $\alpha=\alpha_0+\alpha_2$ with respect to \eqref{dec}.
From \eqref{d} we obtain
\begin{equation}\label{da}
\dd\alpha=\dd\alpha_0+\dd\alpha_2\in (\Lambda^2\mv^*\otimes\Lambda^2\mz^*)\oplus\Lambda^4\mv^*,\end{equation}
and from \eqref{delta} we get 
\begin{equation}\label{dda}\dd^*\alpha=\dd^*\alpha_2\in\Lambda^2\mz^*.
\end{equation}
 
If $\alpha$ is a conformal Killing form on $(\mn,g)$, then by Proposition~\ref{pro:ck}, for all $x\in \mv$ and $z\in \mz$ the following system is satisfied
\begin{equation}\label{ck02}\begin{cases}
\pi_{3,1}(\nabla_x (\alpha_0+\alpha_2))&=\frac1{4}\ x\lrcorner\ \dd\alpha_0-\frac{1}{n-2} x\wedge \dd^*\alpha_2, \\
\pi_{3,3}(\nabla_x \alpha_2)&=\frac1{4}\ x\lrcorner\ \dd\alpha_2,  \\
0&=-\frac{1}{n-2} z\wedge \dd^*\alpha_2,\\
\nabla_z \alpha_2&=\frac1{4}\ z\lrcorner\ \dd\alpha_0,
\end{cases}
\end{equation}
where we recall that $\pi_{k,d}:\Lambda^k\mn^*\lra \Lambda^d\mv^*\otimes \Lambda^{k-d}\mz^*$ denote the projections with respect to the direct sum decomposition \eqref{dec}.
The third equation of the above system shows that if $\dim(\mz)\ne 2$, then $\dd^*\alpha=\dd^*\alpha_2=0$ and thus $\alpha$ is a Killing form. Hence $\mathcal{CK}_{\rm ev}^3(\mn,g)=\mathcal{K}_{\rm ev}^3(\mn,g)$ if $\dim \mz\neq 2$.

Consider now the case $\dim(\mz)= 2$. In this case $\alpha_0\in\Lambda^3\mz^*$ vanishes and thus $\alpha=\alpha_2$, so for every $x\in\mv$ and $z\in \mz$, \eqref{ck02} becomes
\begin{equation}\label{ck02s}\begin{cases}
\pi_{3,1}(\nabla_x \alpha_2)&=-\frac{1}{n-2} x\wedge \dd^*\alpha_2,\\
\pi_{3,3}(\nabla_x \alpha_2)&=\frac1{4}\ x\lrcorner\ \dd\alpha_2,\\
\nabla_z \alpha_2&=0.
\end{cases}
\end{equation}

Let $z_1$, $z_2$ be an orthonormal basis of $\mz$. We denote $A_i:=j(z_i)=-\dd z_i\in \Lambda^2\mv^*\simeq \so(\mv)$, for $i=1,2$,  and write
$\alpha=B_1\wedge z_1+ B_2\wedge z_2$, with $B_1,B_2\in\Lambda^2\mv^*$. 

We have $\pi_{3,1}(\nabla_x \alpha_2)=\nabla_xB_1\wedge z_1+\nabla_x B_2\wedge z_2,$ and 
$$g(\dd^*\alpha_2,z_1\wedge z_2)=g(\alpha_2,-A_1\wedge z_2+z_1\wedge A_2)=-g(B_2,A_1)+g(B_1,A_2),$$
where the scalar products on the right hand side are the 2-form scalar products. 
We thus obtain
\begin{equation}\label{wh}
\dd^*\alpha_2=(-g(B_2,A_1)+g(B_1,A_2))z_1\wedge z_2.
\end{equation}
The first equation in \eqref{ck02s} thus becomes
$$z_2\lrcorner\ \nabla_xB_1-z_1\lrcorner\ \nabla_xB_2=-\frac{1}{n-2}(-g(B_2,A_1)+g(B_1,A_2))x, \qquad\mbox{ for all } x\in \mv.$$
Using \eqref{eq:nabla} we compute $z_2\lrcorner\ \nabla_xB_1=\nabla_x(z_2\lrcorner\ B_1)-(\nabla_xz_2)\lrcorner\ B_1=\frac12B_1(A_2x)$, and similarly 
$z_1\lrcorner\ \nabla_xB_2=\frac12B_2(A_1x)$, whence
$$\frac12B_1(A_2x)-\frac12B_2(A_1x)=-\frac{1}{n-2}(-g(B_2,A_1)+g(B_1,A_2))x, \qquad\mbox{ for all } x\in \mv.$$
Since the 2-form scalar products are half of the endomorphism scalar products, the above equation is equivalent to 
\begin{equation}\label{BABA}
B_1A_2-B_2 A_1=\lambda {\rm Id},
\end{equation}for some $\lambda\in\R$.

Using \eqref{eq:nabla} and the fact that $\dd z_i=-A_i$ for $i=1,2$, the second equation in \eqref{ck02s} becomes
$$-\frac12(B_1\wedge A_1x+B_2\wedge A_2x)=-\frac14 x\lrcorner\ (B_1\wedge A_1+B_2\wedge A_2),  \qquad\mbox{ for all } x\in \mv,$$
or equivalently 
\begin{equation*}\label{ab}B_1\wedge A_1x+B_2\wedge A_2x=B_1x\wedge A_1+B_2x\wedge A_2,  \qquad\mbox{ for all } x\in \mv.\end{equation*}
Contracting this equation with $x$ yields
\begin{equation}\label{BxAx}
B_1x\wedge A_1x+B_2x\wedge A_2x=0,  \qquad\mbox{ for all } x\in \mv.
\end{equation}

Finally, the last equation in \eqref{ck02s} is equivalent to the fact that $B_1$ and $B_2$ commute with $A_1$ and $A_2$.

Assume now that $A_1$ and $A_2$ are linearly independent. Then Proposition 5.4 from \cite{va} shows that $B_1$ and $B_2$ are linear combinations of $A_1$ and $A_2$:
$B_1=a_{11}A_1+a_{12}A_2$, $B_2=a_{21}A_1+a_{22}A_2$, for some $a_{ij}\in \R$, $i,j=1,2$. Reinjecting in the last equation yields $(a_{12}-a_{21})A_1x\wedge A_2x=0$ for every $x\in\mv$, so $a_{12}-a_{21}=0$. This shows that $g(B_2,A_1)=g(B_1,A_2)$, whence $\dd\alpha=\dd^*\alpha_2=0$ by \eqref{wh}, so $\alpha$ is a Killing form. Hence $\mathcal{CK}_{\rm ev}^3(\mn,g)=\mathcal{K}_{\rm ev}^3(\mn,g)$ also in this case.

It remains to consider the case where $A_1$ and $A_2$ are proportional where, by performing a rotation in $\mz$ if necessary, one can assume that $A_2=0$. In this case $(\mn,g)$ is the orthogonal direct sum of the abelian 1-dimensional algebra $\langle z_2\rangle$ and the Heisenberg algebra $\mh_{2q+1}$ of dimension $2q+1:=n-1$. Moreover, from \eqref{BABA} and \eqref{BxAx} we see that $B_1$ is proportional to $A_1$ and $B_2$ is proportional to $A_1^{-1}$. That is, $\alpha$ belongs to the vector space spanned by $A_1 \wedge z_1$ and $A_1^{-1}\wedge z_2$. 

It is easy to check from the above that these two $3$-forms are indeed conformal Killing on $(\mn,g)$. Moreover, $\dd^*( A_1 \wedge z_1 )=0$, $\dd (A_1^{-1}\wedge z_2)=0$ and $\dd^*( A_1^{-1}\wedge z_2)=q\, z_1\wedge z_2$ by  \eqref{d} and \eqref{wh}. 

Therefore $\mathcal{CK}^3_{\rm ev}(\R\oplus\mh_{2q+1},g)$ is spanned by the Killing $3$-form $A_1 \wedge z_1 $ and the strict conformal Killing $3$-form $A_1^{-1}\wedge z_2$ which, being closed, is the Hodge dual of a Killing $2q-1$-form. In particular, $\mathcal K^3(\R\oplus\mh_{2q+1},g)$ is 1-dimensional. This completes the proof.
\end{proof}

\subsection{The case of odd $\mv$-degree}
Now we continue by describing the space $\mathcal{CK}_{\rm odd}^3(\mn,g)$. It follows from \cite[Proposition 5.1]{va} that $\mathcal{K}_{\rm odd}^3(\mn,g)=0$ for any 2-step nilpotent metric Lie algebra $(\mn,g)$. Therefore, every non-zero conformal Killing $3$-form on $(\mn,g)$ of odd $\mv$-degree  is automatically strict.

Let $\alpha$ in $\Lambda_{\rm odd}^3\mn^*$ be a $3$-form of odd $\mv$-degree and consider its decomposition $\alpha=\alpha_1+\alpha_3$ with respect to \eqref{dec}.
From \eqref{d} we obtain
\begin{equation}\label{da1}
\dd\alpha=\dd\alpha_1\in \Lambda^3\mv^*\otimes \mz^*,\end{equation}
and from \eqref{delta} we get 
\begin{equation}\label{dda1}\dd^*\alpha=\dd^*\alpha_3\in\mv^*\otimes \mz^*.
\end{equation}
 
Assume from now on that $\alpha$ is a conformal Killing form. Then, by Proposition~\ref{pro:ck}, for every $x\in \mv$ and $z\in \mz$ the following system is verified
\begin{equation}\label{sxzodd}\begin{cases}\pi_{3,0}(\nabla_x \alpha_1)&=0,\\
\pi_{3,2}(\nabla_x (\alpha_1+\alpha_3))&=\frac1{4}\ x\lrcorner\ \dd\alpha_1-\frac{1}{n-2} x\wedge \dd^*\alpha_3,\\
\nabla_z \alpha_1&=-\frac{1}{n-2} z\wedge \dd^*\alpha_3,\\
\nabla_z \alpha_3&=\frac1{4}\ z\lrcorner\ \dd\alpha_1.\\
\end{cases}
\end{equation}

Let $p$ and $m$ denote the dimensions of $\mv$ and $\mz$ respectively. We fix orthonormal bases $x_1,\ldots,x_p$ of $\mv$ and $z_1,\ldots,z_m$ of $\mz$ and write 
\begin{equation}\label{eq:alpha1}
\alpha_1=\sum_{s,r=1}^{m}\xi_{sr}\wedge z_s\wedge z_r,
\end{equation} 
with $\xi_{sr}=-\xi_{rs}\in\mv$. Using the notation $A_s:=j(z_s)=-\dd z_s\in\Lambda^2\mv^*\simeq\mathfrak{so}(\mv)$, we compute
\begin{equation}\label{da2}\dd\alpha_1=\sum_{s,r=1}^{m}\dd(\xi_{sr}\wedge z_s\wedge z_r)=-2\sum_{s,r=1}^{m}\xi_{sr}\wedge \dd z_s\wedge  z_r=2\sum_{s,r=1}^{m}A_s\wedge\xi_{sr}\wedge z_r,\end{equation}
and for every $s=1,\ldots,m$ and $x\in\mv$,
$$g(\dd^*\alpha_3,x\wedge z_s)=g(\alpha_3,\dd(x\wedge z_s))=g(\alpha_3,x\wedge j(z_s))=g(A_s\lrcorner\ \alpha_3,x),$$
where $A_s\lrcorner\ \alpha_3:=\frac12\sum_{a=1}^p A_s(x_a)\lrcorner\ x_a\lrcorner\ \alpha_3$ is the contraction of $\alpha_3$ with the 2-form $A_s$. The previous formula thus implies
\begin{equation}\label{da3}\dd^*\alpha_3=\sum_{s=1}^m (A_s\lrcorner\ \alpha_3)\wedge z_s.\end{equation}

Since $\nabla_{z_k}=-\frac12j(z_k)_*=-\frac12(A_k)_*$ for every  $k=1,\ldots,m$, we can then write the third and fourth equations in \eqref{sxzodd} as
\begin{equation}\label{s3}\sum_{s,r=1}^{m}(A_k\xi_{sr})\wedge z_s\wedge z_r=\frac{2}{n-2}z_k\wedge\sum_{s=1}^m (A_s\lrcorner\ \alpha_3)\wedge z_s,\qquad \mbox{ for all } k=1,\ldots,m,
\end{equation}
and
\begin{equation}\label{s4}(A_k)_*\alpha_3=\sum_{s=1}^{m}A_s\wedge\xi_{sk},\qquad \mbox{ for all } k=1,\ldots,m.
\end{equation}

It is easy to check that \eqref{s3} is equivalent to 

\begin{equation}\label{s31}A_k\xi_{ij}=0, \mbox{ for all } k\ne i,j,\quad \mbox{ and }\quad A_i\xi_{ij}=-\frac{1}{n-2}A_j\lrcorner\ \alpha_3, \mbox{ for all } i\ne j.\qquad
\end{equation}

After a straightforward computation, it turns out that the first equation in \eqref{sxzodd} is equivalent to 
$$A_k\xi_{ij}+A_i\xi_{jk}+A_j\xi_{ki}=0,\qquad \mbox{ for all } i,j,k,$$
which is clearly a consequence of \eqref{s31}.

We finally interpret the second equation in \eqref{sxzodd}. For $k=1,\ldots,m$ and $x\in\mv$, we have by \eqref{da2}--\eqref{da3}:
$$z_k\lrcorner\ (\frac1{4}\ x\lrcorner\ \dd\alpha_1-\frac{1}{n-2} x\wedge \dd^*\alpha_3)=\frac12\sum_{s=1}^{m}(A_sx\wedge\xi_{sk}+g(x,\xi_{sk})A_s)-\frac{1}{n-2} x\wedge(A_k\lrcorner\ \alpha_3),$$
and 
\begin{eqnarray*}z_k\lrcorner\ \pi_{3,2}(\nabla_x (\alpha_1+\alpha_3))&=&\pi_{2,2}(z_k\lrcorner\ \nabla_x (\alpha_1+\alpha_3))\\
&=&\pi_{2,2}(\nabla_x(z_k\lrcorner\ \alpha_1)-\nabla_xz_k\lrcorner\  (\alpha_1+\alpha_3))\\
&=&\pi_{2,2}(\nabla_x(2\sum_{s=1}^{m}\xi_{sk}\wedge z_s))+\frac12 A_kx\lrcorner\ \alpha_3\\
&=&\sum_{s=1}^{m}A_sx\wedge \xi_{sk}+\frac12 A_kx\lrcorner\ \alpha_3.
\end{eqnarray*}

Consequently, the second equation in \eqref{sxzodd} is equivalent to
\begin{equation}\label{s2}A_kx\lrcorner\ \alpha_3+\sum_{s=1}^{m}(A_sx\wedge\xi_{sk}-g(x,\xi_{sk})A_s)+\frac{2}{n-2} x\wedge(A_k\lrcorner\ \alpha_3)=0,
\end{equation} 
for all $x\in\mv$ and $k=1,\ldots,m$. We denote by $\beta_k:=-\frac{1}{n-2}A_k\lrcorner\ \alpha_3$, so \eqref{s31}--\eqref{s2} read
\begin{equation}\label{s311}A_k\xi_{ij}=\delta_{ik}\beta_j-\delta_{jk}\beta_i,\qquad \mbox{ for all } i,j,k=1,\ldots,m,\qquad
\end{equation}
and
\begin{equation}\label{s21}A_kx\lrcorner\ \alpha_3+\sum_{s=1,s\ne k}^{m}(A_sx\wedge\xi_{sk}-g(x,\xi_{sk})A_s)-2 x\wedge\beta_k=0,
\end{equation} 
for all $ x\in\mv$ and $k=1,\ldots,m$. We thus obtain that \eqref{sxzodd} is equivalent to the system of equations given by \eqref{s4}, \eqref{s311} and \eqref{s21}.

Taking the interior product with $A_kx$ in \eqref{s21} and using \eqref{s311} yields
\begin{equation}\label{s22}\sum_{s=1,s\ne k}^{m}\left(g(A_kx,A_sx)\xi_{sk}-g(\beta_s,x)A_sx-g(x,\xi_{sk})A_sA_kx\right)+2 g(\beta_k,A_kx)x=0,
\end{equation} 
for all $x\in\mv$ and $k=1,\ldots,m$. Recall that for every $j\ne k$ we have $A_k\xi_{jk}=-\beta_j$ and $A_j\xi_{jk}=\beta_k$, and for  all $s\ne j,k$ we have $A_s\xi_{jk}=0$. Thus \eqref{s22} applied to $x=\xi_{jk}$ reads
$$-g(\beta_j,\beta_k)\xi_{jk}-g(\beta_j,\xi_{jk})\beta_k+\sum_{s=1,s\ne k}^{m}g(\xi_{jk},\xi_{sk})A_s\beta_j-2g(\beta_k,\beta_j)\xi_{jk}=0,$$
and using that $g(\beta_j,\xi_{jk})=-g(A_k\xi_{jk},\xi_{jk})=0$ (since $A_k$ is skew-symmetric), we obtain
\begin{equation}\label{s26}3g(\beta_j,\beta_k)\xi_{jk}=\sum_{s=1,s\ne k}^{m}g(\xi_{jk},\xi_{sk})A_s\beta_j,\quad\mbox{ for all } x\in\mv,\mbox{ and } j\ne k.\end{equation} 
Taking the scalar product with $\xi_{jk}$ in \eqref{s26} yields
$$3g(\beta_j,\beta_k)||\xi_{jk}||^2=-\sum_{s=1,s\ne k}^{m}g(\xi_{jk},\xi_{sk})g(\beta_j,A_s\xi_{jk})=-||\xi_{jk}||^2g(\beta_j,\beta_k),\quad\mbox{ for all }  j\ne k.$$

If $g(\beta_j,\beta_k)\ne 0$, the above relation gives $\xi_{jk}=0$, so $\beta_k=A_j\xi_{jk}=0$, a contradiction. Thus 
\begin{equation}\label{s23}g(\beta_j,\beta_k)=0,\quad\mbox{ for all }  j\ne k.
\end{equation} 

Taking the scalar product with $\xi_{ij}$ in \eqref{s26} for some $i\ne j,k$ and using \eqref{s23} gives
$$0=\sum_{s=1,s\ne k}^{m}g(\xi_{jk},\xi_{sk})g(A_s\beta_j,\xi_{ij})=-g(\xi_{jk},\xi_{ik})||\beta_j||^2,$$
whence 
\begin{equation}\label{s27}g(\xi_{jk},\xi_{ik})\beta_j=0,\qquad  \mbox{ for all } i\ne j\ne k\ne i.
\end{equation} 

We next take $x=x_a$ in \eqref{s22} and sum over $a=1,\ldots,p$ to obtain
$$\sum_{s=1,s\ne k}^{m}\left(-\tr(A_sA_k)\xi_{sk}-A_s\beta_s-A_sA_k\xi_{sk}\right)-2A_k\beta_k=0, \quad\mbox{ for all }  k=1,\ldots,m,$$
whence, using \eqref{s311} again,
\begin{equation}\label{s24}2A_k\beta_k+\sum_{s=1,s\ne k}^{m}\tr(A_sA_k)\xi_{sk}=0, \quad\mbox{ for all }  k=1,\ldots,m.
\end{equation} 

Applying $A_k$ to this equation, taking the scalar product with $\beta_k$ and using \eqref{s23}, yields,
\begin{eqnarray*}||A_k\beta_k||^2&=&-g(A_k^2\beta_k,\beta_k)=\frac12\sum_{s=1,s\ne k}^{m}\tr(A_sA_k)g(A_k\xi_{sk},\beta_k)\\&=&-\frac12\sum_{s=1,s\ne k}^{m}\tr(A_sA_k)g(\beta_s,\beta_k)=0,\end{eqnarray*}
thus showing that
 \begin{equation}\label{s25}A_k\beta_k=0, \quad\mbox{ for all }  k=1,\ldots,m.
\end{equation} 

The following technical result shows that if $\alpha$ is non-zero, then at least two of the vectors $\beta_i$,  $i=1, \ldots, m$ are non-zero.

\begin{lm}\label{lm:bvanish}
If $\alpha$ is non-vanishing, then there exist $i, j\in \{1, \ldots, m\}$  with $i\neq j$ such that $\beta_i\neq 0\neq \beta_j$. In particular, the center of $\mn$ has dimension $m\geq 2$.
\end{lm}
\begin{proof}
Suppose for a contradiction that $\alpha\neq 0$ and there exists $\ell\in \{1, \ldots, m\}$ such that $\beta_i=0$ for all $i\in \{1,\ldots,m\}\bs\{\ell\}$. 
 
By relabelling the subscripts, we may assume that $\ell=1$, that is, $\beta_i=0$ for all $i=2,\ldots m$. We claim that $\beta_1=0$.

If $m=1$, then \eqref{s25} together with \eqref{int} readily imply $\beta_1=0$, thus proving our claim in this case. Assume now that $m\ge 2$.

From \eqref{s311} we obtain $A_k\xi_{ij}=0$ for all $i,j\ge 2$ and for every $k$, so by \eqref{int}, $\xi_{ij}=0$ for every $i,j\ge 2$. For $k=2$, \eqref{s22} reads
\begin{equation}\label{s29}g(A_2x,A_1x)\xi_{12}-g(\beta_1,x)A_1x-g(x,\xi_{12})A_1A_2x=0,\qquad   \mbox{ for all } x\in\mv.
\end{equation} 

Polarizing this equality yields
\begin{eqnarray}\label{s40}\qquad 0&=&(g(A_2x,A_1y)+g(A_2y,A_1x))\xi_{12}-g(\beta_1,x)A_1y\\
&&-g(x,\xi_{12})A_1A_2y-g(\beta_1,y)A_1x-g(y,\xi_{12})A_1A_2x, \qquad   \mbox{ for all } x,y\in\mv.\nonumber
\end{eqnarray} 
Applying this to $y=\xi_{12}$ and using \eqref{s311} and \eqref{s25}, we obtain
$$0=-g(\beta_1,\xi_{12})A_1x-||\xi_{12}||^2A_1A_2x, \qquad   \mbox{ for all } x\in\mv.$$
Since $g(\beta_1,\xi_{12})=-g(A_2\xi_{12},\xi_{12})=0$, this equation reads
\begin{equation}\label{eq:xiAA}
0=||\xi_{12}||^2A_1A_2.
\end{equation}
If $\xi_{12}=0$, then $\beta_1=-A_2\xi_{12}=0$ proving our claim. Otherwise, if $\xi_{12}\ne 0$, then  \eqref{eq:xiAA} implies $A_1A_2=0$, so reinjecting this into \eqref{s40} gives 
$$-g(\beta_1,x)A_1y-g(\beta_1,y)A_1x=0, \qquad   \mbox{ for all } x,y\in\mv.$$
Taking $y=\beta_1$ in this last equation and using \eqref{s25} we obtain that either $A_1=0$ or $\beta_1=0$. We finish the proof of the claim by showing that $A_1= 0$ also implies $\beta_1=0$.

Indeed, if $A_1=0$ then \eqref{s4} for $k=1$ yields $0=\sum_{s=1}^{m}A_s\wedge\xi_{s1},$ so 
$$\sum_{s=1}^{m}(A_sx\wedge\xi_{s1}+g(x,\xi_{s1})A_s)=0, \qquad   \mbox{ for all } x\in\mv.$$

This, together with \eqref{s2} for $k=1$ gives
$$\sum_{s=1}^{m}g(x,\xi_{s1})A_s=0, \qquad   \mbox{ for all } x\in\mv.$$
Applying this equality to $\xi_{21}$, taking $x=\xi_{21}$ and using \eqref{s311}, yields $||\xi_{21}||^2\beta_1=0$. Since $\xi_{21}\neq 0$, we finally get $\beta_1=0$.

We  thus proved our claim: $\beta_i=0$ for all $i=1, \ldots,m$. Hence, by \eqref{dda1} and \eqref{da3}, and taking into account that $\beta_s=-\frac{1}{n-2}A_s\lrcorner\ \alpha_3$, we get $\dd^*\alpha=0$ and therefore $\alpha$ is a Killing $3$-form of odd $\mv$-degree. On the other hand from \cite[Proposition 5.1]{va} we know that $\mathcal{K}_{\rm odd}^3(\mn,g)=0$. This contradicts the assumption $\alpha\neq 0$ thus concluding the proof.
\end{proof}

We are now ready to describe the space of conformal Killing $3$-forms of odd $\mv$-degree for every 2-step nilpotent metric Lie algebra $(\mn,g)$. This will be done by considering the different possibilities for the dimension of the center of $\mn$.

\begin{pro}\label{pro:zleq2} On a $2$-step nilpotent metric Lie algebra $(\mn,g)$ with $\dim\mz\leq 2$, every conformal Killing $3$-form of odd $\mv$-degree vanishes.
\end{pro}
\begin{proof} Let $(\mn,g)$ be a 2-step nilpotent metric Lie algebra. 
If $\dim\mz=1$, then the result follows from Lemma~\ref{lm:bvanish}. So we consider the case $m=\dim \mz= 2$.

Assume that there exists a non-zero conformal Killing $3$-form $\alpha$ of odd $\mv$-degree.  We write as before $\alpha=\alpha_3+\alpha_1$ with $\alpha_3\in\Lambda^3\mv^*$ and $\alpha_1$ as in \eqref{eq:alpha1}, for a fixed orthonormal basis $z_1,z_2$ of $\mz$. To simplify the notation, we denote by $\xi:=\xi_{12}=-\xi_{21}$. From \eqref{s311} we obtain
\begin{equation}\label{s50} A_1\xi=\beta_2,\qquad A_2\xi=-\beta_1.
\end{equation}
By Lemma~\ref{lm:bvanish}, $\beta_1\neq 0\neq \beta_2$ so \eqref{s50} implies $\xi\ne 0$. By rescaling $\alpha$ if necessary, we may assume that $||\xi||=1$.
Formula \eqref{s21} for $k=1$ becomes
\begin{equation}\label{s433}
A_1x\lrcorner\ \alpha_3-A_2x\wedge\xi+g(x,\xi) A_2-2 x\wedge\beta_1=0,\quad\mbox{ for all }  x\in\mv.
\end{equation}
Taking $x=\beta_1$ in this equation and using \eqref{s25} yields $A_2\beta_1\wedge\xi=0$. Thus $A_2\beta_1=a\xi$ where the real number $a$ is given by
$a=g(\xi,A_2\beta_1)=-g(\beta_1,A_2\xi)=||\beta_1||^2$. Consequently, 
\begin{equation}\label{s441}
A_2\beta_1= ||\beta_1||^2\xi.
\end{equation}

Now we make the interior product in \eqref{s433}  with $A_1x$ and then take the inner product with $\beta_1$. Using \eqref{s25} again we obtain
\[
g(A_1A_2x,x)g(\xi,\beta_1)+g(\xi,A_1x)g(A_2x,\beta_1)+g(x,\xi)g(A_2A_1x,\beta_1)=0, \quad\mbox{ for all }  x\in \mv.
\]
This equation, in view of \eqref{s50} and \eqref{s441}, is equivalent to 
\begin{equation}
\label{s46}
2||\beta_1||^2g(x,\xi)g(x,\beta_2)=0, \quad\mbox{ for all }  x\in \mv.
\end{equation}

Since $\xi$ is non-zero, this implies that either $\beta_1=0$ or $\beta_2=0$, which contradicts Lemma~\ref{lm:bvanish}. This concludes the proof.
\end{proof}

Since the case where $\dim\mz=3$ is more involved, we treat next the case $\dim \mz\geq 4$.

\begin{pro}\label{pro:zgeq4} On a $2$-step nilpotent metric Lie algebra $(\mn,g)$ with $\dim\mz\geq 4$, every conformal Killing $3$-form of odd $\mv$-degree vanishes.
\end{pro} 
\begin{proof} Let $(\mn,g)$ be  a 2-step nilpotent metric Lie algebra  with center of dimension $m\geq 4$. Let $\alpha$ be a conformal Killing $3$-form on $(\mn,g)$ of odd $\mv$-degree. We write $\alpha=\alpha_3+\alpha_1$, with $\alpha_3\in\Lambda^3\mv^*$ and $\alpha_1$ as in \eqref{eq:alpha1}, where $z_1,\ldots, z_m$ is a fixed orthonormal basis of $\mz$. We claim that:
\begin{equation}\label{s43}
g(\xi_{ij},\beta_k)=0, \quad\mbox{ for all } \, i,j, k=1, \ldots, m.
\end{equation}
Indeed, since $m\geq 4$, given $i,j,k\in \{1, \ldots, m\}$ we can find $r\in \{1, \ldots,m\}$ such that $r\notin\{i,j,k\}$. Using \eqref{s311}, we obtain 
\begin{equation*}
g(\xi_{ij},\beta_k)=g(\xi_{ij},A_r\xi_{rk})=-g(A_r\xi_{ij},\xi_{rk})=0,
\end{equation*}as claimed.

Suppose for a contradiction that $\alpha\neq 0$. Then, by Lemma~\ref{lm:bvanish}, there exist $i\neq j$ such that  $\beta_i\neq 0\neq \beta_j$. Let $k$ be such that $i\neq k\neq j$. Taking the scalar product of the expression in \eqref{s22} with $\xi_{ik}$, and using \eqref{s25}, we get
\begin{equation}\label{s41}
\sum_{s=1,s\ne k}^{m}g(A_kx,A_sx)g(\xi_{sk},\xi_{ik})+g(\beta_i,x)g(\beta_k,x)=0,\quad\mbox{ for all }  x\in \mv.
\end{equation} 
Since $\beta_i$ is nonzero, \eqref{s41} and \eqref{s27} account to
\begin{equation}\label{s42}
g(A_kx,A_ix)||\xi_{ik}||^2+g(\beta_i,x)g(\beta_k,x)=0,\quad\mbox{ for all }  x\in \mv.
\end{equation} 
Polarizing this equation one gets for all $x,y\in \mv$:
\begin{equation*}
(g(A_kx,A_iy)+g(A_ky,A_ix))||\xi_{ik}||^2+g(\beta_i,x)g(\beta_k,y)+g(\beta_i,y)g(\beta_k,x)=0.
\end{equation*} 
We take $x=\xi_{kj}$ in this equation and use \eqref{s311} and \eqref{s43} to obtain
\begin{equation}\label{s431}
g(\beta_j,A_iy)||\xi_{ik}||^2=0,\quad\mbox{ for all }  y\in \mv.
\end{equation}
Notice that $\xi_{ik}\neq 0$ since $\beta_i=-A_{k}\xi_{ik}$ is non-zero. So from \eqref{s431} we get $0=A_i\beta_j=A_i^2\xi_{ij}$. Taking the scalar product with $\xi_{ij}$ and using the skew-symmetry of $A_i$ yields $0=A_i\xi_{ij}=\beta_j$, which is a contradiction.
\end{proof}

Finally, we consider the case where $\dim\mz=3$ and start by exhibiting an example of a 2-step nilpotent metric Lie algebra carrying a non-zero conformal Killing 3-form of odd $\mv$-degree. 

Let $\mn_{3,2}$ be the 6-dimensional free 2-step nilpotent Lie algebra on $3$-generators, which admits a basis $v_1,v_2,v_3,u_1, u_2,u_3$ satisfying the Lie bracket relations
\begin{equation}\label{eq:n32b}
[v_1,v_2]=u_3,\quad [v_2,v_3]=u_1,\quad [v_3,v_1]=u_2.
\end{equation}
Clearly the center of $\mn_{3,2}$ is spanned by $u_i$, $i=1, 2,3$.

We consider the one parameter family of metrics $g_\lambda$, with $\lambda\in \R^+$, on $\mn_{3,2}$ defined by the fact that $v_1,v_2,v_3,u_1/\lambda, u_2/\lambda,u_3/\lambda$ is a $g_\lambda$-orthonormal basis. Denoting by $z_i:=u_i/\lambda$ for $i=1,2,3$, one can easily check that the endomorphisms $j(z_i)$ satisfy $[j(z_i),j(z_j)]=\lambda j(z_k)$, for every  even permutation $(i,j,k) $ of $\{1,2,3\}$. Therefore, by \cite[Proposition 5.8]{va}, the space of Killing $3$-forms on  $(\mn_{3,2},g_\lambda)$ is 1-dimensional and spanned by
\[
\eta:= z_1\wedge v_2\wedge v_3+z_2\wedge v_3\wedge v_1 +z_3\wedge v_2\wedge v_1+2 z_1\wedge z_2 \wedge z_3.
\]In particular, $\mathcal{K}_{\rm odd}^3(\mn_{3,2},g_\lambda)=0$ and $\mathcal{K}_{\rm ev}^3(\mn_{3,2},g_\lambda)={\rm span}\{\eta\}$.

Since the Hodge dual of a Killing form is a conformal Killing form, we then obtain that $\ast \eta\in \mathcal{CK}^3(\mn_{3,2},g_\lambda)$, independently of the orientation of $(\mn_{3,2},g_\lambda)$. Moreover, since $\eta$ is of even $\mv$-degree, we actually get that $\ast \eta$ belongs to $\mathcal{CK}_{\rm odd}^3(\mn_{3,2},g_\lambda)$, so this space is non-trivial.

It is easy to see that $(\mn_{3,2},g_\lambda)$ is the only 2-step nilpotent metric Lie algebra of dimension 6 where this construction provides non-trivial conformal Killing $3$-forms of odd $\mv$-degree. Indeed, from \cite[Theorem 5.14]{va}, non-zero Killing 3-forms only exist on $\R^3\oplus \mh_3$, $\mh_3\oplus \mh_3$, $\R\oplus \mh_5$ and $\mn_{3,2}$, and they have even $\mv$-degree by \cite[Theorem 5.1]{va}. For the Hodge dual to have odd $\mv$-degree, one thus needs $\mv$ to be odd-dimensional, which only happens on $\mn_{3,2}$.

We shall now prove that $(\mn_{3,2},g_\lambda)$ are, up to isomorphism, the only 2-step nilpotent metric Lie algebras with $\dim \mz=3$, admitting strict conformal Killing 3-forms of odd $\mv$-degree and also that $\mathcal{CK}_{\rm odd}^3(\mn_{3,2},g_\lambda)$ is spanned by $\ast \eta$.

\begin{pro}\label{pro:3formsodd}On a $2$-step nilpotent metric Lie algebra $(\mn,g)$ with $\dim\mz=3$, every conformal Killing $3$-form of odd $\mv$-degree vanishes, except when $(\mn,g)=(\mn_{3,2},g_\lambda)$ for some $\lambda\in \R^+$ where $\mathcal{CK}_{\rm odd}^3(\mn_{3,2},g_\lambda)$ is $1$-dimensional.
\end{pro}

\begin{proof}
Let $(\mn,g)$ be a 2-step nilpotent metric Lie algebra  with center of dimension $m=3$. 

Let $\alpha$ be a non-zero conformal Killing $3$-form $\alpha$ on $(\mn,g)$ of odd $\mv$-degree. We write as before $\alpha=\alpha_3+\alpha_1$ with $\alpha_3\in\Lambda^3\mv^*$ and 
\begin{equation}\label{eq:alpha11}
\alpha_1=\sum_{s,r=1}^{3}\xi_{sr}\wedge z_s\wedge z_r,
\end{equation}
where $z_1,z_2,z_3$ is an orthonormal basis of $\mz$ and $\xi_{sr}=-\xi_{rs}\in\mv$. We keep using the notation $A_i:=j(z_i)$ and $\beta_i:=-\frac{1}{n-2}A_i\lrcorner\ \alpha_3$.

For each $i,j=1,2,3$ with $i\neq j$, $\beta_i$ and $\beta_j$ cannot vanish simultaneously by Lemma~\ref{lm:bvanish}. Since $\beta_j=A_i\xi_{ij}$ and $\beta_{i}=-A_j\xi_{ij}$ by \eqref{s311}, this shows that the vectors $\xi_{ij}$ are non-zero for every $i\neq j$ in $\{1,2,3\}$. 

Moreover, since \eqref{s27} is symmetric in $i\neq j$, and either $\beta_i$ or $\beta_j$ are non-zero, we obtain
\begin{equation}\label{s500}
g(\xi_{jk},\xi_{ik})=0, \quad \mbox{ for all } i\neq j\neq k\neq i.
\end{equation} 

In particular, the vectors $\xi_{12},\xi_{23}$ and $\xi_{31}$ span a 3-dimensional vector subspace, which we will call $E$, of $\mv$. 

Let  us fix a permutation $(i,j,k)$ of $\{1, 2,3\}$. We will show that $\beta_k$ is collinear to $\xi_{ij}$.
To this purpose, suppose that $\beta_k\neq  0$ (since otherwise the result is trivial), and let us take $x=\xi_{ij}$ in \eqref{s21}. By \eqref{s311} and \eqref{s500}, we obtain
\begin{equation}\label{s100} 0=\sum_{s=1,s\ne k}^{3}(A_s\xi_{ij}\wedge\xi_{sk}-g(\xi_{ij},\xi_{sk})A_s)-2\xi_{ij}\wedge\beta_k=
\beta_j\wedge\xi_{ik}-\beta_i\wedge\xi_{jk}-2\xi_{ij}\wedge\beta_k.
\end{equation} 
 Taking the interior product with $\beta_k$ in \eqref{s100}, and using that $g(\beta_k,\xi_{ik})=g(A_{i}\xi_{ik},\xi_{ik})=0$  and $g(\beta_i,\beta_j)=0$ by \eqref{s23}, we get 
\begin{equation}\label{s51}
g(\xi_{ij},\beta_k)\beta_k=||\beta_k||^2\xi_{ij}.
\end{equation}
Since $\beta_k\neq 0$, this equation and the fact that  $\xi_{ij}\neq0$ imply $g(\xi_{ij},\beta_k)\neq 0$. Hence $\beta_k$ is a multiple of $\xi_{ij}$   by \eqref{s51} as claimed.

In particular, we obtain that $\beta_1,\beta_2,\beta_3$ belong to the 3-dimensional subspace $E\subset \mv$ spanned by the vectors $\xi_{ij}$, so by \eqref{s311} we see that $E$, and hence also its orthogonal $E^\bot$, are invariant under the skew-symmetric endomorphisms $A_i$, $i=1,2,3$.

From now on, we will assume that $(i,j,k)$ is an even permutation of $\{1,2,3\}$. Consider an orthonormal basis $f_1,f_2,f_3$ of $E$ such that $\xi_{ij}=a_kf_k$, for some $a_k\in \R^*$. Since $\beta_k$ is a multiple of $\xi_{ij}$ we then have
\begin{equation}\label{s54}
\beta_k=b_{k} f_k, \quad \mbox{ for some } b_k\in \R. 
\end{equation}

For each $i=1,2,3$, we have $a_i A_i f_i=A_{i}\xi_{jk}=0$ by \eqref{s311}, so 
\begin{equation}
\label{s53} A_i|_{E}=\lambda_i f_j\wedge f_k, \quad \mbox{ for some }\lambda_i\in \R.
\end{equation}
Note that $\lambda_i\ne0$ since otherwise $\beta_j=A_i\xi_{ij}$ and $\beta_{k}=A_i\xi_{ik}$ would both vanish, which we have seen that it was impossible.
By changing the sign of the elements in the basis $z_1,z_2,z_3$ of $\mz$ if necessary, we may thus assume that $\lambda_i>0$ for all $i=1,2,3$.

Now using \eqref{s311}, \eqref{s54} and \eqref{s53} we obtain, for every even permutation $(i,j,k)$ of $\{1,2,3\}$:
\begin{eqnarray*}
&b_jf_j=\beta_j=A_i\xi_{ij}=\lambda_i a_k\, f_k\lrcorner\ (f_j\wedge f_k)=-\lambda_i a_k f_j.\\
&b_if_i=\beta_i=-A_j\xi_{ij}=-\lambda_j a_k\, f_k\lrcorner\ (f_k\wedge f_i)=-\lambda_j a_k f_i.
\end{eqnarray*}
This implies \begin{equation}\label{s90}a_k=-b_j/\lambda_i=-b_i/\lambda_j,\end{equation}  for every even permutation $(i,j,k)$ of $\{1,2,3\}$.  Therefore there exists $c\in\R$ such that
\begin{equation}\label{s55}
\lambda_ib_i=\lambda_jb_j=\lambda_kb_k=c.
\end{equation}

Since $A_jf_j=0$ by \eqref{s53} and $g(f_j,\xi_{jk})=a_ig(f_j,f_i)=0$, \eqref{s21} for $x=f_j$ gives
\begin{equation*}
A_kf_j\lrcorner\ \alpha_3+A_if_j\wedge\xi_{ik}-g(f_j,\xi_{ik})A_i-2 f_j\wedge\beta_k=0.
\end{equation*} 
From this equation and using \eqref{s53} again we obtain
\begin{equation}\label{s56}
-\lambda_k f_i\lrcorner\ \alpha_3+A_if_j\wedge\xi_{ik}-g(f_j,\xi_{ik})A_i-2 f_j\wedge\beta_k=0.
\end{equation}
Since $A_if_j\wedge\xi_{ik}-2 f_j\wedge\beta_k$ and $A_i$ are elements of $\Lambda^2E^*\oplus \Lambda^2(E^\bot)^*\subset \Lambda^2\mv^*$, the above relation shows that $x\lrcorner\ \alpha_3\in \Lambda^2E^*\oplus \Lambda^2(E^\bot)^*$ for every $x\in E$, thus implying that $\alpha_3\in \Lambda^3E^*\oplus (E^*\otimes \Lambda^2(E^\bot)^*)$. Furthermore, since $\dim E=3$, one can write
\begin{equation}\label{s57}
\alpha_3=\mu f_1\wedge f_2\wedge f_3+\gamma,\quad \mbox{ for some }\mu\in\R\;\mbox{ and } \;\gamma \in E^*\otimes \Lambda^2(E^\bot)^*.
\end{equation}

Projecting the equality in \eqref{s56} onto $\Lambda^3E^*$, and using \eqref{s53} together with $\xi_{ik}=-a_jf_j$, we obtain
\begin{equation*}
-\mu \lambda_k f_j\wedge f_k-a_j\lambda_if_k\wedge f_j+a_j\lambda_if_j\wedge f_k-2 b_k f_j\wedge f_k=0.
\end{equation*}
In view of  \eqref{s90}, this is equivalent to 
\begin{equation}
\label{s58}
\mu \lambda_k +4 b_k=0,\quad\mbox{ for all }  k\in\{1,2,3\}.
\end{equation}

Equations \eqref{s58} and \eqref{s55} imply 
$\mu \lambda_k^2=-4b_k\lambda_k=-4c$ so from this, together with \eqref{s90}, we get $
\lambda_k^2=-4\frac{c}{\mu}$ for $k=1,2,3$. 
In particular, since $\lambda_s>0$ for all $s=1,2,3$, we obtain $\lambda_i=\lambda_j=\lambda_k=:\lambda$. Consequently, $f_1,f_2,f_3$ is an orthonormal basis of $E\subset \mv$ whose Lie brackets, by \eqref{s53}, verify
\begin{equation}\label{eq:brackE}
[f_i,f_j]=\sum_{r=1}^3 g(A_rf_i,f_j) z_r=\lambda z_k,
\end{equation}
whenever $(i,j,k)$ is an even permutation of $\{1,2,3\}$.

Before proceeding, notice that we can now express the parameters above in terms of $c$ and $\lambda$. Indeed,  \eqref{s58}, \eqref{s55} and \eqref{s90}, yield
\begin{equation}\label{eq:coeff}
\mu=-4\frac{c}{\lambda^2},\quad b_k=\frac{c}{\lambda},
 \quad a_k=-\frac{c}{\lambda^2},\quad  \mbox{ for  all } k=1,2,3.
\end{equation}
By \eqref{eq:alpha11}, \eqref{eq:coeff}, \eqref{s57} and the  fact that $\xi_{ij}=a_kf_k$ for every even permutation $(i,j,k)$ of $\{1,2,3\}$, we obtain that the component of $\alpha$ in $\Lambda^3(E\oplus \mz)^*$ is
\begin{equation}\label{eq:alphcan}
\alpha_{\Lambda^3(E\oplus \mz)^*}=-2\frac{c}{\lambda^2}(f_3\wedge z_1\wedge z_2+f_2\wedge z_3\wedge z_1+f_1\wedge z_2\wedge z_3+2f_1\wedge f_2\wedge f_3).
\end{equation}
From \eqref{eq:brackE} and \eqref{eq:alphcan}, it is easy to check that $(\mn,g)$ has a Lie subalgebra isometrically isomorphic to $(\mn_{3,2},g_\lambda)$. Indeed, $\varphi:\mn_{3,2}\lra \mn$ defined by $\varphi(v_i)=f_i$, $\varphi(u_i)=\lambda z_i$ is an injective Lie algebra morphism, which is an isometry on its image $E\oplus \mz$.

We will show below that $E^\bot=0$ and once this is proven, the result will follow. Indeed, on the one hand $E^\bot=0$ implies that $\varphi$ is actually onto, so $(\mn,g)$ is isometrically isomorphic to $(\mn_{3,2},g_\lambda)$. On the other hand, $E^\bot=0$ also implies that $\alpha=\alpha_{\Lambda^3(E\oplus \mz)^*}$ and it is given by the right hand side of \eqref{eq:alphcan}. By using \eqref{d} and \eqref{s53} one can check that $\dd\alpha=0$ so, for any possible orientation of $(\mn,g)$, $\ast \alpha$ is a coclosed conformal Killing $3$-form (i.e. Killing). This shows that $\mathcal{CK}^3_{\rm odd}(\mn,g)=\ast \mathcal{K}^3_{\rm ev}(\mn,g)
$ which is 1-dimensional by \cite[Proposition 5.8]{va}.
 
 The rest of the proof aims to show that $E^\bot=0$.  Let $x\in E$ and $y\in E^ \bot$ and $k\neq j$ in $\{1,2,3\}$. Taking the interior product in \eqref{s21} with $A_jy$ we obtain
\begin{equation*}
A_jy\lrcorner A_kx\lrcorner \alpha_3+\sum_{s=1,s\neq k}^3 (g(A_sx,A_jy)\xi_{sk}-g(\xi_{sk},A_jy)A_sx-g(x,\xi_{sk})A_sA_jy)=2 A_jy\lrcorner (x\wedge \beta_k).
\end{equation*}
Since $A_jy\in E^\bot$, this reduces to
\begin{equation}\label{s60}
A_jy\lrcorner A_kx\lrcorner \alpha_3-\sum_{s=1,s\neq k}^3g(x,\xi_{sk})A_sA_jy=0.
\end{equation}
Similarly, by interchanging $x$ with $y$ and $j$ with $k$ when using \eqref{s21}, we get
\begin{equation}\label{s61}
A_kx\lrcorner\ A_jy\lrcorner\ \alpha_3-\sum_{s=1,s\neq j}^3g(\xi_{sj},A_kx)A_sy-2 A_kx\lrcorner (y\wedge \beta_j)=0.
\end{equation}
Adding up \eqref{s60} and \eqref{s61} we get that for every $x\in E$, $y\in E^\bot$, 
\begin{equation*}
-\sum_{s=1,s\neq k}^3g(x,\xi_{sk})A_sA_jy-\sum_{s=1,s\neq j}^3g(\xi_{sj},A_kx)A_sy+2 g(A_kx,\beta_j) y=0.
\end{equation*}
This equation holds for every distinct $j,k\in\{1,2,3\}$. Denoting by $i$ the remaining subscript and developing the terms in each sum, we get
\begin{equation}\label{s64}
-g(x,\xi_{ik})A_iA_jy-g(x,\xi_{jk})A_j^2y-g(\xi_{kj},A_kx)A_ky-2 g(A_k\beta_j,x) y=0.
\end{equation}

Assume first that $(i,j,k)$ is an even permutation of $\{1,2,3\}$. Then  $\xi_{ik}=-a_jf_j$ and  $\xi_{jk}=-\xi_{kj}=a_if_i$, so \eqref{s64} becomes
\begin{equation}\label{s65}
a_j g(x,f_j)A_iA_jy-a_ig(x,f_i)A_j^2y+b_j g(f_j,x)A_ky+2 \lambda_k b_j g(x,f_i) y=0.
\end{equation}
For $x=f_i$ this expression reads $a_i A_j^2y=2 \lambda_k b_j y$, and taking $x=f_j$ in the same equation, we obtain 
$a_j A_iA_jy+b_j A_ky=0$.
These formulas, together with \eqref{eq:coeff}, imply
\begin{equation}\label{s67}
A_i^2|_{E^\bot}=-2\lambda^2 {\rm Id}_{E^\bot},\mbox{ for }i=1,2,3\mbox{ and } A_iA_j|_{E^\bot}=\lambda A_k|_{E^\bot} \mbox{ for }(i,j,k) \mbox{ even}.
\end{equation}

Now, if $(i,j,k)$ is an odd permutation of $\{1,2,3\}$, we have $\xi_{ik}=a_jf_j$, so taking $x=f_j$ in \eqref{s64} yields $a_jA_iA_jy+b_j A_ky=0$ for every $y\in E^\bot$, which reads
\begin{equation}\label{s68}
A_iA_j|_{E^\bot}=-\lambda A_k|_{E^\bot} \mbox{ for }(i,j,k) \mbox{ odd}.
\end{equation}
We can thus conclude, from \eqref{s67} and \eqref{s68}, that $A_iA_j|_{E^\bot}=-A_jA_i|_{E^\bot}$ for all $i\neq j$.

Finally, for $(i,j,k)$ an even permutation of $\{1,2,3\}$ we have, on the one hand 
$$A_i^2A_j|_{E^\bot}=-2\lambda^2 A_j|_{E^\bot}$$ by \eqref{s67}, and on the other hand, $$A_i^2A_j|_{E^\bot}=A_iA_iA_j|_{E^\bot}=-A_iA_jA_i|_{E^\bot}=-\lambda A_iA_k|_{E^\bot}=\lambda^2 A_j|_{E^\bot}.$$ Since $\lambda\ne 0$, this implies that $A_j|E^\bot=0$ for every $j$, so ${E^\bot}=0$ by \eqref{int}.
\end{proof}

Finally, we are in the position to state the main result of the section, which follows immediately from the decompositions in \eqref{ckdec}, \eqref{kdec} and Propositions~\ref{pro:3formsevenv},~\ref{pro:zleq2},~\ref{pro:zgeq4} and~\ref{pro:3formsodd}.

\begin{teo}\label{teo:3forms} 
For any metric $2$-step nilpotent Lie algebra $(\mn,g)$, $\mathcal{CK}^3(\mn,g)=\mathcal{K}^3(\mn,g)$, except when $\mn=\mathfrak{h}_{2q+1}\oplus\mathbb{R}$ and $g$ is any metric on $\mh_{2q+1}\oplus\R$, and when  $(\mn,g)=(\mn_{3,2},g_\lambda)$ for some $\lambda\in\R^+$ where, in both cases,  $\mathcal{CK}^3(\mn,g)$ is $2$-dimensional and $\mathcal{K}^3(\mn,g)$ is $1$-dimensional.
\end{teo}

\begin{remark} Revisiting (the proofs of) Propositions~\ref{pro:3formsevenv} and~\ref{pro:3formsodd}, one can see that when the space of conformal Killing $3$-forms on a 2-step nilpotent metric Lie algebra does not coincide with the space of Killing $3$-forms, it can be described more precisely as follows:
\begin{eqnarray*}
&\mathcal{CK}^3(\R\oplus\mh_{2q+1},g)=\mathcal{CK}_{\rm ev}^3(\R\oplus\mh_{2q+1},g)=\mathcal{K}_{\rm ev}^3(\R\oplus\mh_{2q+1},g)\oplus \ast \mathcal{K}_{\rm ev}^{2q-1}(\R\oplus\mh_{2q+1},g),\\
&\mathcal{CK}^3(\mn_{3,2},g_\lambda)=\mathcal{CK}_{\rm ev}^3(\mn_{3,2},g_\lambda)\oplus \mathcal{CK}_{\rm odd}^3(\mn_{3,2},g_\lambda) =\mathcal{K}_{\rm ev}^3(\mn_{3,2},g_\lambda)\oplus \ast \mathcal{K}_{\rm ev}^{3}(\mn_{3,2},g_\lambda).
\end{eqnarray*}
\end{remark}

\end{document}